\documentclass[11pt]{amsart}

 \usepackage{amsmath,amsthm,amsfonts,amssymb,verbatim}
 \usepackage{url}
 \usepackage{graphicx}
 \usepackage[all]{xy}
 \usepackage{wrapfig}
 \usepackage{pinlabel}
 \usepackage{subfigure}
 
\newcommand{\Br}{\mathbf{\Sigma}}

\newcommand{\bu}{\bullet}

\newcommand{\bZ}{\mathbb{Z}}

\newcommand{\si}{\sigma}

\newcommand{\HFhat}{\widehat{\operatorname{HF}}}

\newcommand{\rk}{\operatorname{rk}}

\newtheorem{theorem}{Theorem}

\newtheorem{corollary}[theorem]{Corollary}
\newtheorem{proposition}[theorem]{Proposition}
\newtheorem{remark}[theorem]{Remark}
\newtheorem{lemma}[theorem]{Lemma}

\title[Genus one open books with non-left-orderable fundamental group]{Genus one open books with non-left-orderable fundamental group}
\date{September 22, 2011}

\author[Yu Li]{Yu Li}
\thanks{Yu Li was supported by a CSST undergraduate summer internship}
\address{School of Mathematics, Nankai University, Tianjin 300071 China}
\email{LiYu173@mail.nankai.edu.cn}

\author[Liam Watson]{Liam Watson}
\thanks{Liam Watson was partially supported by an NSERC postdoctoral fellowship}
\address{Department of Mathematics, UCLA, 520 Portola Plaza, Los Angeles, CA 90095}
\email{lwatson@math.ucla.edu}

\begin{document}

\begin{abstract}
Let $Y$ be a closed, connected, orientable three-manifold admitting a genus one open book decomposition with one boundary component. We prove that if $Y$ is an L-space, then the fundamental group of $Y$ is not left-orderable. This answers a question posed by John Baldwin.
\end{abstract}

\maketitle

A group is left-orderable if it admits a strict total ordering of its elements that is invariant under multiplication on the left. By convention, the trivial group is not left-orderable; for definitions and background relevant to this paper see \cite{BGW2011,BRW2005}. Let $Y$ be a closed, connected, irreducible, orientable three-manifold. It has been conjectured that $Y$ is an L-space if and only if $\pi_1(Y)$ is not left-orderable \cite[Conjecture 3]{BGW2011}. We verify one direction of this conjecture when $Y$ admits a genus one open book decomposition with a single boundary component.

\begin{theorem}\label{thm:main}
Let $Y$ be a closed, connected, orientable three-manifold admitting a genus one open book decomposition with one boundary component. If $Y$ is an L-space then $\pi_1(Y)$ is not left-orderable.
\end{theorem}

Recall that an L-space is a rational homology sphere $Y$ with simplest possible Heegaard Floer homology, in the sense that $\rk\HFhat(Y)=|H_1(Y;\bZ)|$ \cite{OSz2005-lens}. Baldwin gives a complete classification of L-spaces among manifolds admitting a genus one open book decomposition with a single boundary component \cite[Theorem 4.1]{Baldwin2008}. In particular, these may be identified as the two-fold branched covers of (the closures of) an explicit family of three-braids \cite{Baldwin2008} (see Section \ref{sec:Baldwin}). Given this classification, the proof of Theorem \ref{thm:main} follows from a study of the fundamental groups of the particular two-fold branched covers that arise. Having first collected some relevant known results concerning this family (see Proposition \ref{prp:types-2-and-3} and Proposition \ref{prp:type-1-d=0}), the main step in our proof focusses on a particular sub-class of three-braids (see Proposition \ref{prp:type-1-rest}). We make use of a group presentation for the two-fold branched cover derived from the white graph of a diagram for the branch set due to Greene \cite{Greene2008, Greene2011} (see Section \ref{sec:Greene}), and show that when this presentation is associated with a graph of a particular form, the resulting group cannot be left-orderable (see Proposition \ref{prp:cycle}).

Before turning to the requisite material for the proof of Theorem \ref{thm:main} recall that, in the context of Heegaard Floer homology and two-fold branched covers, a natural extension of the class of alternating links is given by quasi-alternating links (see \cite[Definition 3.1]{OSz2005-branch}). We remark that, as Baldwin gives a complete classification of quasi-alternating links that are the closures of three-braids \cite[Theorem 8.7]{Baldwin2008}, it follows immediately from Theorem \ref{thm:main} that:

\begin{corollary}If $L$ is a quasi-alternating link with braid index at most three then the two-fold branched cover of $L$ has non-left-orderable fundamental group. \end{corollary}

It is natural to posit, in light of \cite[Conjecture 3]{BGW2011}, that the fundamental group of the two-fold branched cover of any quasi-alternating link is not left-orderable. This is known for non-split alternating links  \cite[Theorem 8]{BGW2011}, and further infinite families of examples may be obtained by combining the examples of quasi-alternating links that arise in \cite{Watson2009} (as the branch sets of certain L-spaces obtained via Dehn surgery) with results about Dehn surgery and non-left-orderability established in \cite{CW2010,CW2011}. In particular, combine \cite[Theorem 28]{CW2010} with \cite[Theorem 5.1]{Watson2009} and/or  \cite[Theorem 3]{CW2011} with \cite[Theorem 6.1]{Watson2009}. Other examples are studied in work of Ito \cite{Ito2011} and Peters \cite{Peters2009}. The study of quasi-alternating links in this context is closely related to the study of Dehn surgery questions pertaining to left-orderability. In particular, properties of L-spaces suggest that left-orderability of the fundamental group of a 3-manifold should be well behaved under Dehn surgery (see \cite[Question 8]{CW2010} and \cite[Question 3.1]{Greene2011}, for example).

\subsection*{Acknowledgements} This work formed part of an undergraduate research project undertaken while the first author was a participant in the CSST summer program at UCLA. We thank John Baldwin for suggesting the question answered by Theorem \ref{thm:main}. 

\section{On Baldwin's classification}\label{sec:Baldwin}

Murasugi gives a complete classification of three-braids up to conjugacy \cite{Murasugi1974} (compare \cite[Theorem 2.2]{Baldwin2008}). As a strict subset of these, Baldwin suggests the following families (of conjugacy classes) of three-braids:

\begin{itemize}
\item[(1)] $h^d\si_1\si_2^{-a_1}\cdots\si_1\si_2^{-a_n}$ where $a_i\ge0$, $a_j\ne0$ for some $0\le j\le n$ and $d=-1,0,1$
\item[(2)] $h^d\si_2^m$ where $d=\pm 1$
\item[(3)] $h^d\si_1^m\si_2^{-1}$ where $m=-1,-2,-3$ and $d=-1,0,1,2$.
\end{itemize}

In these classes, $h=(\si_2\si_1)^3$ denotes the full-twist on three strands. With these three families in hand, we summarize Baldwin's classification of L-spaces as follows:

\begin{theorem}[Baldwin {\cite[Section 2 and Theorem 4.1]{Baldwin2008}}]\label{thm:Baldwin}
A genus one open book decomposition with a single boundary component is an L-space if and only if it is the two-fold branched cover of the closure of a braid of type (1),  (2) or (3).\end{theorem}

Towards the proof of Theorem \ref{thm:main} we make two observations. 

\begin{proposition}\label{prp:types-2-and-3}If $Y$ is the two-fold branched cover of a braid of type (2) or (3) then $\pi_1(Y)$ is finite and hence not left-orderable.\end{proposition}

\begin{proof}First consider the closures of the braids of type (2). In this case, consulting \cite[Proof of Theorem 8.7 part {\bf II}]{Baldwin2008}, the branch set $\overline{h^{\pm1}\si_2^m}$ in question is a pretzel knot encodes a Seifert structure with base orbifold $S^2(2,2,2+m)$ (see in particular  \cite[Figure 9]{Baldwin2008}). It follows from Scott's classification of Seifert structures \cite{Scott1983} (compare \cite[Proposition 2.3]{OSz2005-lens}) that the two-fold branched coved in question admits elliptic geometry, and hence the fundamental group is finite as claimed.

For braids of type (3), we may appeal to \cite[Proof of Theorem 8.7 part {\bf III}]{Baldwin2008}. First note that in the case $d=0$ and $d=1$ the branch sets in question are the (two-bridge) torus knots $T(2,m)$ and $T(2,m+4)$, respectively. As the resulting two-fold branched covers must be lens spaces, the corresponding fundamental groups are finite cyclic. Up to mirrors, the remaining branch sets may be viewed as the closures of $d^5=d^4\si_1^2,d^4\si_1,d^4$ (these are the links $10_{124}$, $L_{9n12}$, $8_{19}$, respectively). The corresponding two-fold branched covers may also be obtained (up to orientation reversal) by $+1$-, $+2$- and $+3$-surgery on the right-hand trefoil, respectively, giving Seifert structures with base orbifold $S^2(2,3,n)$ for $n=5,4,3$. As in the case of type (2), these admit elliptic geometry and hence the fundamental group is finite as claimed.  \end{proof}

\begin{remark} For Seifert fibred spaces, L-space is equivalent to non-left-orderable fundamental group \cite[Theorem 4]{BGW2011}. Therefore, combined with the fact that the branch sets of type (2) and (3) have two-fold branched covers that are L-spaces, it is sufficient for our purposes to simply note that all of these manifolds are Seifert fibred. It seems interesting to note, however, the stronger statement that all of these groups are indeed finite. \end{remark}

\begin{proposition}\label{prp:type-1-d=0}If $Y$ is the two-fold branched cover of a braid of type (1) with $d=0$ then the branch set is alternating and hence $\pi_1(Y)$ is not left-orderable.\end{proposition}

\begin{proof} The braids of type (1) with $d=0$ are alternating on inspection of the diagram (compare \cite[Proof of Theorem 8.7 part {\bf I}]{Baldwin2008}); the result then follows from \cite[Theorem 8]{BGW2011}. \end{proof}

Thus, to prove Theorem \ref{thm:main}, it remains to show that the two-fold branched cover of the closure of a braid of type (1) with $d=\pm 1$ has non-left-orderable fundamental group.

\section{On Greene's presentation}\label{sec:Greene}

For our purposes, a convenient description of the fundamental group of the two-fold branched cover of a link $L$ is given as follows. Let $\Gamma$ be the white graph of a checkerboard colouring of the link $L$. Decorate the edges of $\Gamma$ according to the convention in Figure \ref{fig:conventions} and distinguish an arbitrary vertex $r$ (the {\em root}). Consider the group \[G_\Gamma=\langle x_1,\ldots, x_n | r_1,\ldots,r_n, x_r \rangle\]
where the generators $x_i$ are in one-to-one correspondence with the vertices of $\Gamma$, and the relations are specified as follows. At each edge $(x_i,x_j)$ incident to a vertex $x_i$ define the word $w_j^i=(x_j^{-1}x_i)^{\epsilon(x_i,x_j)}$, where $\epsilon(x_i,x_j)$ is the sign on the edge. Then $r_i$ is the product of the $w_j^i$ read in counter-clockwise order around a small loop centred at the vertex $x_i$.

\begin{figure}[ht!]
\begin{center}
\labellist
\pinlabel $+$ at 160 380
\pinlabel $-$ at 610 380
\endlabellist
\includegraphics[scale=0.27]{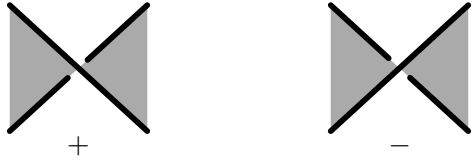}
\caption{Sign conventions at a crossing given a colouring of a knot diagram.}\label{fig:conventions}
\end{center}
\end{figure}

\begin{theorem}[Greene {\cite[Proposition 1.1]{Greene2011}}]Let $\Br(L)$ denote the two-fold branched cover of $L$, and fix a white graph $\Gamma$ for some diagram of $L$. Then $\pi_1(\Br(L))\cong G_\Gamma$.\end{theorem}

To state our next result, we will make use of slightly different graph, $\widetilde{\Gamma}$. This is obtained from the signed white graph $\Gamma$ by removing the root vertex $r$, and decorating each remaining vertex $i$ with an integer specifying the number of edges between $i$ and $r$, with sign. It is immediate that $\widetilde{\Gamma}$ retains enough information to reproduce the presentation $G_\Gamma$, provided we record the region in the plane that contained the root $r$. For our purposes, the root vertex will be in the unbounded region of the plane.

\begin{figure}[ht!]
\begin{center}
\labellist
\pinlabel $\cdots$ at 115 152
\pinlabel $\cdots$ at 115 -2
\pinlabel \rotatebox{-52}{$\cdots$} at 33 18
\pinlabel \rotatebox{52}{$\cdots$} at 177 50

\pinlabel $\cdots$ at 365 152
\pinlabel $\cdots$ at 365 -2
\pinlabel \rotatebox{-52}{$\cdots$} at 283 18
\pinlabel \rotatebox{52}{$\cdots$} at 427 50

\small

\pinlabel $a_0$ at -8 75
\pinlabel $a_n$ at 196 75

\pinlabel $a_1$ at 47 -6
\pinlabel $a_{n-1}$ at 152 -6

\pinlabel $x_1$ at 290 155
\pinlabel $x_{m-1}$ at 402 155

\pinlabel $y_0\!=\!x_0$ at 229 75
\pinlabel $y_{c_n}\!=\!x_m$ at 462 75

\pinlabel $y_1$ at 253 55
\pinlabel $y_2$ at 262 40
\pinlabel $y_{c_1}$ at 297 -7
\pinlabel $y_{c_1+1}$ at 326 -7

\pinlabel $y_{c_{n-1}}$ at 402 -7
\pinlabel $y_{c_{n-1}+1}$ at 422 16
\pinlabel $y_{c_{n-1}+2}$ at 433 32

\pinlabel $-$ at 20 120
\pinlabel $-$ at 70 155
\pinlabel $-$ at 168 120

\pinlabel $-$ at 270 120
\pinlabel $-$ at 320 155
\pinlabel $-$ at 418 120
\endlabellist
\includegraphics[scale=0.8]{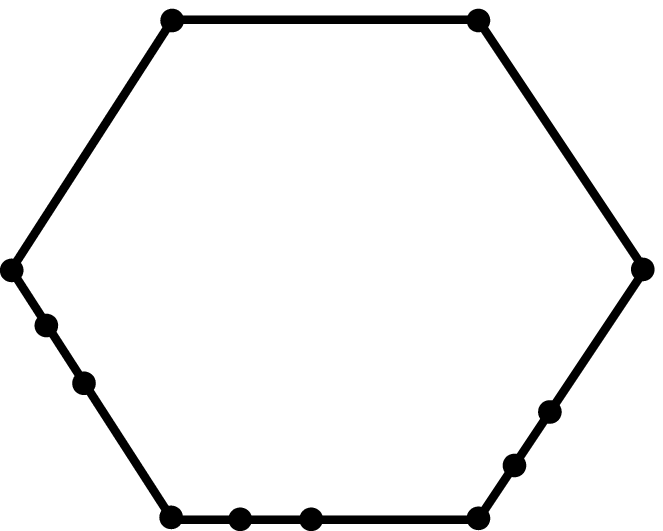}
\qquad\qquad
\includegraphics[scale=0.8]{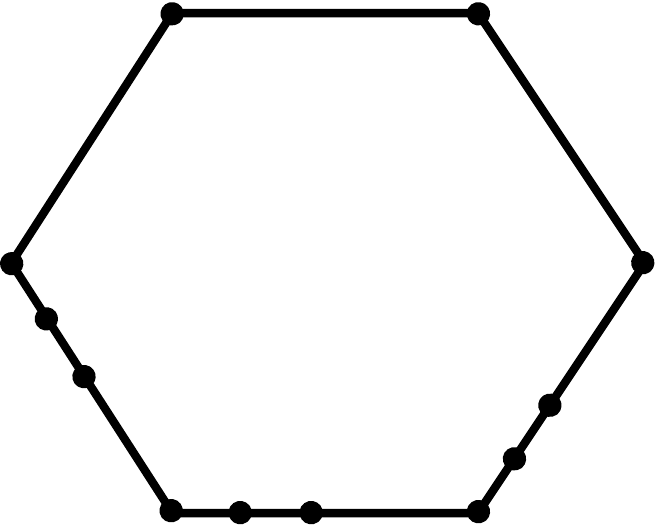}
\caption{The form of $\widetilde{\Gamma}$ considered in the hypothesis of Proposition \ref{prp:cycle}. On the left $\widetilde{\Gamma}$ is shown, where unmarked vertices should be labelled with $0$ and unmarked edges should be labelled with $+$. On the right the vertices are labelled with their corresponding generator in $G_\Gamma$.}\label{fig:cycle}
\end{center}
\end{figure}

Now consider the case wherein $\widetilde{\Gamma}$ consists of a single cycle of the form shown in Figure \ref{fig:cycle}. Then $G(\Gamma)$ is generated by $c_n+m+1$ elements denoted \[x_1,\ldots, x_{m-1},y_0,y_1,\ldots,y_{c_n},z\] where $z$ is the generator associated with the vertex $r$ removed from $\Gamma$ to form $\widetilde{\Gamma}$. Note that the generators $y_{c_k}$ for $0\le k \le n$ are precisely those corresponding the the vertices labelled with the $a_k>0$. Set $c_k=b_1+\cdots + b_k$ where $b_k-1$ is the number of vertices {\em between} the vertices labelled $a_{k-1}$ and $a_k$, and let $c_0=0$. For future reference, we denote the relations by $r(g)$ for a given vertex labelled with a generator $g$. Assuming that $n>0$, the relations of $G_\Gamma$ have the following form:
\begin{align*}
r(x_i) &= (x_{i+1}^{-1}x_i)^{-1}(x_{i-1}^{-1}x_i)^{-1} \quad{\rm where}\quad 0<i<m\\
r(y_0) & =  (x_1^{ - 1}{y_{0})^{-1}y_{0}}^{{a_0}}(y_{1}^{ - 1}{y_{0}})\\
r(y_{c_k}) &= (y_{{c_k} - 1}^{ - 1}{y_{{c_k}}})y_{{c_k}}^{{a_k}}(y_{{c_k} + 1}^{ - 1}{y_{{c_k}}}) \quad{\rm where}\quad 0<k<n \\
r(y_{c_n}) & =(y_{c_n-1}^{ - 1}{y_{c_n}})  y_{c_n}^{{a_n}}(x_{m-1}^{ - 1}y_{c_n})^{-1} \\
r(y_j) &=  (y_{i+1}^{-1}y_i)(y_{i-1}^{-1}y_i) \quad{\rm where}\quad i\ne c_k \quad{\rm for}\quad 0<k<n
\end{align*}
Recall that there are two additional relations $z$ (i.e. $z=1$) and $r(z)=y_{{c_n}}^{{-a_n}}\cdots y_1^{{-a_1}}y_{{c_0}}^{{-a_0}}$. Of course, the latter is equivalent to $y_0^{{a_0}}y_{{c_1}}^{{a_1}} \cdots y_{{c_n}}^{{a_n}}=1$.

\begin{proposition}\label{prp:cycle} If $\widetilde{\Gamma}$ is of the form shown in Figure \ref{fig:cycle} and either $m > 1$ or $m = 1$ and ${a_0} ,{a_n} > 1$, then $G_\Gamma$ is not left-orderable.\end{proposition}

The proof, occupying the remainder of this section, is established by way of a series of lemmas.

\begin{lemma}\label{lem:WLOG} If $G_\Gamma$ is left-orderable then without loss of generality we may assume that, relative to any given left-order, $y_0 < 1 < y_{c_n}$.\end{lemma}
\begin{proof} Suppose that $G_\Gamma$ is left-orderable, and fix a left order $<$. As observed by Greene \cite[Proof of Theorem 2.1]{Greene2011}, $G_\Gamma$ must contain at least one generator that is non-trivial and  larger (or equal to) all other generators relative to $<$, since otherwise the group is trivial (recall that the trivial group is not left-orderable by convention). Similarly, $G_\Gamma$ must contain a generator that is non-trivial and smaller (or equal to) all other generators relative to $<$.

Next consider a vertex with associated generator $g_i$ for which every incident edge is labelled $+$. By definition, the associated relation at this vertex is $\prod_j g_j^{-1}g_i$, where $j$ runs over vertices adjacent to the $g_i$ vertex. If we assume that $g_i$ is a largest element among the generators, then $g_j\le g_i$ for any $j$, which implies that $1\le g_j^{-1}g_i$. If $1< g_j^{-1}g_i$ we contradict $\prod_j g_j^{-1}g_i=1$; in particular, any of the $g_j$ must be a largest element also (compare \cite[Proof of Theorem 2.1]{Greene2011}). The same observation holds for least elements among the generators; a similar argument applies for all incident edges labelled $-$ (in both cases). As a result, without loss of generality, we may assume that least and greatest elements among the generators correspond to vertices  with incident edges that are not all labeled with the same sign. 

In the present setting, we have that $y_0$ and $y_{c_n}$ are the candidates for least and greatest elements among the generators. Note that these generators must have opposite sign relative to $<$, otherwise we contradict $y_0^{{a_0}}y_{{c_1}}^{{a_1}} \cdots y_{{c_n}}^{{a_n}}=1$. We conclude that the only possibilities are either $y_0<1<y_{c_k}$ or $y_{c_k}<1<y_0$; by symmetry (i.e. by passing to the opposite order) we may restrict attention to the former, concluding the proof. \end{proof}

\begin{lemma}\label{lem:x} The elements $x_i\in G_\Gamma$ may be rewritten ${x_i} = {({x_1}x_0^{ - 1})^{i - 1}}{x_1}$ for all $0\le i \le m$.
\end{lemma}
\begin{proof}The statement clearly holds for $i=0,1$. When $i=2$ consider the relation $x_0^{-1}x_1x_2^{-1}x_1$ so that $x_2= x_1x_0^{-1}x_1$, verifying the case $i=2$. For induction, suppose the result holds for $0\le i\le k$, and consider the relation $x_{k-1}^{-1}x_kx_{k+1}^{-1}x_k$ so that \[{x_{k + 1}} = {x_k}x_{k - 1}^{ - 1}{x_k} = {x_i} = {({x_1}x_0^{ - 1})^{k - 1}}{x_1}{({({x_1}x_0^{ - 1})^{k - 2}}{x_1})^{ - 1}}{({x_1}x_0^{ - 1})^{k - 1}}{x_1} = {({x_1}x_0^{ - 1})^k}{x_1}\] as claimed.\end{proof}

\begin{lemma}\label{lem:y} The elements $y_{c_k}\in G_\Gamma$ may be rewritten
\[{y_{{c_k}}} = {({y_{{c_{k - 1}} + 1}}y_{{c_{k - 1}}}^{ - 1})^{{b_k} - 1}}{y_{{c_{k - 1}} + 1}}\] for all  $0 < k \le n$, or
\[{y_{{c_k}}} = {({y_{{c_{k + 1}} - 1}}y_{{c_{k + 1}}}^{ - 1})^{{b_{k + 1}} - 1}}{y_{{c_{k + 1}} - 1}} = {y_{{c_{k + 1}} - 1}}{(y_{{c_{k + 1}}}^{ - 1}{y_{{c_{k + 1}} - 1}})^{{b_{k + 1}} - 1}}\] for all $0 \le k < n$. More generally, for $c_{k-1}\le i \le c_k$ we have\[{y_i} = {({y_{{c_{k - 1}} + 1}}y_{{c_{k - 1}}}^{ - 1})^{i - {c_{k - 1}} - 1}}{y_{{c_{k - 1}} + 1}}\]
and similarly, for $c_{k-1}\le i \le c_k$ we have\[{y_i} = {({y_{{c_k} - 1}}y_{{c_k}}^{ - 1})^{{c_k} - i - 1}}{y_{{c_k} - 1}} = {y_{{c_k} - 1}}{(y_{{c_k}}^{ - 1}{y_{{c_k} - 1}})^{{c_k} - i - 1}}.\]
\end{lemma}
\begin{proof} Identical to that of Lemma \ref{lem:x}.\end{proof}

\begin{lemma}\label{lem:left}Every element $y_{c_k}$ may be represented as a word in the group elements $y_0$ and $x_1y_0^{a_0-1}$ for $0\le k \le n$. \end{lemma}
\begin{proof}
We prove a stronger statement: ${y_{{c_k}}}$ and ${y_{{c_k} + 1}}y_{{c_k}}^{ - 1}$ can be represented as a word in the group elements $y_0$ and $x_1y_0^{a_0-1}$ for $0\le k \le n$.

Note that when $k=0$ the first claim holds trivially. Next consider the relation\[r(y_0)=y_1^{-1}y_0(x_1^{-1}y_0)^{-1}y_0^{a_0}=y_1^{-1}x_1y_0^{a_0}\] so that $y_1=x_1y_0^{a_0}$. Therefore \[{y_1}y_0^{ - 1} = {x_1}y_0^{{a_0}}y_0^{ - 1} = {x_1}y_0^{{a_0} - 1}\] and the second claim holds for $k = 0$ as well.

For induction, assume that the conclusion holds for all $0 \le i \le k - 1$. We have that \[{y_{{c_k}}} = {({y_{{c_{k - 1}} + 1}}y_{{c_{k - 1}}}^{ - 1})^{{b_k} - 1}}{y_{{c_{k - 1}} + 1}} = {({y_{{c_{k - 1}} + 1}}y_{{c_{k - 1}}}^{ - 1})^{{b_k}}}{y_{{c_{k - 1}}}}\] from Lemma \ref{lem:y}, so the claim for ${y_{{c_k}}}$ holds. On the other hand, consider the relation \[r(y_{c_k})=y_{c_k+1}^{-1}y_{c_k}y_{c_k-1}^{-1}y_{c_k}y_{c_k}^{a_k},\] hence\[{y_{{c_k} + 1}}y_{{c_k}}^{ - 1} = {y_{{c_k}}}y_{{c_k} - 1}^{ - 1}y_{{c_k}}^{{a_k}} = {y_{{c_{k - 1}} + 1}}y_{{c_{k - 1}}}^{ - 1}y_{{c_k}}^{{a_k}}.\] In combination with the inductive hypothesis, the claim holds for ${y_{{c_k} + 1}}y_{{c_k}}^{ - 1}$.\end{proof}

\begin{lemma}\label{lem:x_1} Relative to any left-ordering of $G_\Gamma$, $x_1$ must be a positive element.\end{lemma}
\begin{proof}Consider the relation $y_{c_n}^{-a_n}y_{c_{n-1}}^{-a_{n-1}}\cdots y_1^{-a_1}y_0^{-a_0}$, which implies that \[y_0^{a_0}y_1^{a_1}\cdots y_{c_{n-1}}^{a_{n-1}}y_{c_n}^{a_{n}}=1.\]By Lemma \ref{lem:left}, $y_{c_k}^{a_k}$ may be expressed as a word in the elements $y_0$ and $x_1y_0^{a_0-1}$ for $0\le k \le n$, denoted $w_k=w_k(y_0,x_1y_0^{a_0-1})$. In particular, \[w_0w_1\cdots w_{n-1}w_n=1.\] By Lemma \ref{lem:WLOG}, we may assume that $y_0<1$, thus  $x_1y_0^{a_0-1}$ must be a positive element (if not, a product of negative elements is 1, a contradiction). Now  $1<x_1y_0^{a_0-1}$ implies that $1<x_1$ when $a_0=1$, or $x_1^{-1}<y_0^{a_0-1}$ when $a_0>1$. The former implies that $x_1^{-1}$ is negative, hence $x_1$ is positive as claimed. \end{proof}

\begin{lemma}\label{lem:right}Every element $y_{c_k}$ may be represented as a word in the group elements ${y_{{c_n}}}$ and $y_{{c_n}}^{{a_n} - 1}{x_{m - 1}}$ for $0\le k \le n$. \end{lemma}
\begin{proof}Similar to that of Lemma \ref{lem:left} and left to the reader.\end{proof}

\begin{proof}[Proof of Proposition \ref{prp:cycle}] Suppose that $G_\Gamma$ is left-orderable. By Lemma \ref{lem:WLOG} we may assume, without loss of generality, that $y_0$ is negative and $y_{c_n}$ is positive. There are two cases to consider according to the hypothesis.

{\bf Case 1:} $m > 1$

Write $y_{{c_n}}^{{a_n} - 1}{x_{m - 1}} = y_{{c_n}}^{{a_n} - 1}{({x_1}y_0^{ - 1})^{m - 2}}{x_1}$ using Lemma \ref{lem:x}. Recall that  ${x_1} $ is positive by Lemma \ref{lem:x_1} and $y_{c_n}$ and $y_0^{-1}$ are positive by assumption. Therefore $y_{{c_n}}^{{a_n} - 1}{x_{m - 1}}$ is positive, as a product of positive elements. As a result, using Lemma \ref{lem:right} we can express $y_0$ as a product of positive elements, a contradiction.

{\bf Case 2:} $m = 1$ and ${a_0},{a_n} > 1$

In this case $y_{c_n} = x_1$ and $x_{m - 1} = y_0$. By Lemma \ref{lem:left}, we must  have ${x_1}y_0^{{a_0} - 1}>1$ (otherwise $x_1$, a positive element, may be written as a product of negative elements). Therefore $x_1y_0\ge{x_1}y_0^{{a_0} - 1}>1$ since $y_0$ is negative.  Now $y_{{c_n}}^{{a_n} - 1}{x_{m - 1}} = x_1^{{a_n} - 1}{y_0} = x_1^{a_n-2}(x_1y_0) > 1$ and we have a contradiction, in view of Lemma \ref{lem:right},  as before.\end{proof}

\section{Completing the proof of Theorem \ref{thm:main}} As observed in Section \ref{sec:Baldwin} it remains to consider braids of type (1)  when $d=\pm1$.

\begin{proposition}\label{prp:type-1-rest}If $Y$ is the two-fold branched cover of a braid of type (1) with $d=\pm1$ then $\pi_1(Y)$ is not left-orderable.\end{proposition}

\begin{proof} We proceed considering the two cases separately; each case reduces to an application of  Proposition \ref{prp:cycle}.

\begin{figure}[ht!]
\begin{center}
\labellist
\pinlabel $\bu$ at 185 30
\small
\pinlabel $r$ at 194 24
\pinlabel $\underbrace{\phantom{aaaaaaaaa}}$ at 120 56
\pinlabel $a_0$ at 120 38
\pinlabel $\underbrace{\phantom{aaaaaa}}$ at 240 56
\pinlabel $a_1$ at 240 38
\pinlabel $\underbrace{\phantom{aaaaaaa}}$ at 343 56
\pinlabel $a_3$ at 343 38

\pinlabel $\overbrace{\phantom{aaa}}$ at 183 156
\pinlabel $b_1$ at 183 176
\pinlabel $\overbrace{\phantom{aaaaaa}}$ at 292 156
\pinlabel $b_2$ at 292 176

\pinlabel $\overbrace{\phantom{aaaaaaaaaa}}$ at 206 290
\pinlabel $m$ at 206 307
\endlabellist
\includegraphics[scale=0.5]{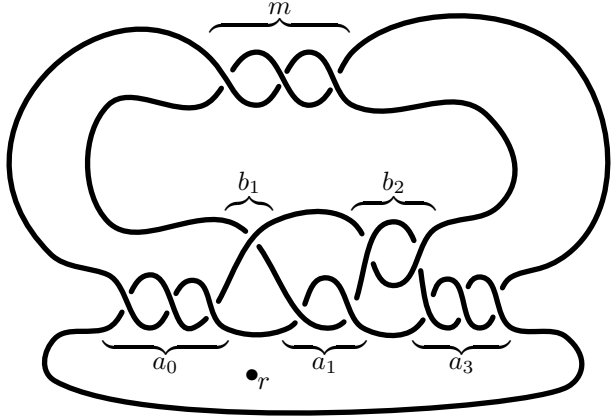}
\caption{An example of a branch set formed from the closure of a  braid of the type used in the proof of Proposition \ref{prp:type-1-rest}, with $m,n=3$. The single vertex $r$ indicates the white region corresponding to the root vertex of $\Gamma$, removed to form $\widetilde{\Gamma}$. }\label{fig:example}
\end{center}
\end{figure}

{\bf Case 1:} $d=1$

Revisiting the form of type (1) braids we have \[h\si_1\si_2^{-a_1}\cdots\si_1\si_2^{-a_n}\] where $a_i\ge0$ and $a_j\ne0$ for at least one value $j$. Up to conjugation, this braid is equivalent to \[h\si_1^k\si_2^{-a_1'}\cdots\si_1\si_2^{-a_n'}\] where now $k,a_1',a_n'>0$, and $a_j'\ge 0$ for $0<j<n$. 

Notice that in this alternate expression the case $n=1$ can arise, in which case the braid in question takes the form $h\si_1^k\si_2^{-a}$ where $k,a>1$. Now recalling that $d=(\si_2\si_1)^3= \si_2\si_1^2\si_2\si_1^2$ we have 
\[h\si_1^k\si_2^{-a} = \si_2\si_1^2\si_2\si_1^m\si_2^{-a}=
\begin{cases}
\si_1^{m+2}\si_2& a=1\\
\si_1^2\si_2^m & a=2 \\
\si_2^m\si_1\si_2^{-a-2}\si_1 & a>2 \\
\end{cases}\] up to conjugation, where $k+2=m>2$. The cases $a=1$ and $a=2$ give rise to the branch sets $T(2,m+2)$ and $T(2,2)\#T(2,a+2)$, respectively, so the fundamental groups of the corresponding two-fold branched covers contain torsion and cannot be left-orderable. On the other hand, the closure of the braid $\si_2^m\si_1\si_2^{-a-2}\si_1$ gives a diagram with white graph satisfying the hypothesis of Proposition \ref{prp:cycle} (where the unbounded region is shaded black in the checkerboard colouring). We remark that, since this particular branch set is a pretzel knot, the two-fold branched cover is a Seifert fibred L-space. As a result, the desired conclusion may also be obtained from \cite[Theorem 4]{BGW2011}.

Now assume that $n>1$. Up to conjugation,  we have the braid 
\begin{align*}
& h\si_1^k\si_2^{-a_1'}\cdots\si_1\si_2^{-a_n'} \\
&= (\si_2\si_1^2\si_2\si_1^2)\si_1^k\si_2^{-a_1'}\cdots\si_1\si_2^{-a_n'} \\
&= \si_2\si_1^2\si_2 \si_1^m\si_2^{-a_1'}\cdots\si_1\si_2^{-a_n'} \\
&= \si_2^m\si_1\si_2^{1-a_1'}\cdots\si_1\si_2^{1-a_n'}\si_1
\end{align*} where $k+2=m>2$. Renaming constants, this braid may be expressed as \[\si_2^m\si_1^{a_0}\si_2^{-b_1}\si_1^{a_1}\cdots\si_2^{-b_n}\si_1^{a_n}\] where $a_i> 0$ for $0\le i \le n$, $b_i>0$ for  $1\le i \le n$ and $m>2$. An example is given in Figure \ref{fig:example}; the associated white graph (shading the unbounded region black) satisfies the hypothesis of Proposition \ref{prp:cycle}.

{\bf Case 2:} $d=-1$

As above, up to conjugation, consider the braids
\begin{align*}
&h^{-1}\si_1\si_2^{-a_1}\cdots\si_1\si_2^{-a_n}\\
&= \si_2^{-1}\si_1^{-1}\si_2^{-1}\si_1^{-1}\si_2^{-1}\si_1^{-1}\si_1\si_2^{-a_1}\cdots\si_1\si_2^{-a_n}\\
&=\si_2^{-1}\si_1^{-1}\si_2^{-1}\si_1^{-1}\si_2^{-a_1-1}\cdots\si_1\si_2^{-a_n}\\
&=\si_1^{-1}\si_2^{-1}\si_1^{-1}\si_2^{-a_1-1}\cdots\si_1\si_2^{-a_n-1}\\
&=\si_2^{-1}\si_1^{-1}\si_2^{-1}\si_2^{-a_1-1}\cdots\si_1\si_2^{-a_n-1}\\
&=\si_1^{-1}\si_2^{-a_1-2}\cdots\si_1\si_2^{-a_n-2}
\end{align*}
where $a_i\ge0$ and $a_j\ne0$ for at least one value $j$. Notice that if $n=1$ then $a_1\ne0$ and the braid in question is $\si_1^{-1}\si_2^{-a_1-4}$ so that the relevant branch set is $T(2,a_1)$. Otherwise, $n>1$ and renaming constants as before we have the braid \[\si_1^{-1}\si_2^{-a_0}\si_1^{b_1}\si_2^{-a_1}\cdots\si_1^{b_n}\si_2^{-a_n}\]  where $a_i> 0$ for $0\le i \le n$ and $b_i>0$ for  $1\le i \le n$. Notice in particular that $a_1,a_n>1$. Now up to exchanging $\si_1 \leftrightarrow \si_2$ and taking mirrors (the former is cosmetic; the latter results in an orientation reversing homeomorphism in the two-fold branched cover) this braid is equivalent to 
\[\si_2\si_1^{a_0}\si_2^{-b_1}\si_1^{a_1}\cdots\si_2^{-b_n}\si_1^{a_n}\] as in the case $d=1$ (this time with $m=1$ and $a_0,a_n>1$), so that the associated white graph satisfies the hypothesis of Proposition \ref{prp:cycle}. \end{proof}

Now combining Proposition \ref{prp:types-2-and-3}, Proposition \ref{prp:type-1-d=0} and Proposition \ref{prp:type-1-rest} with Theorem \ref{thm:Baldwin} proves Theorem \ref{thm:main}.

\bibliographystyle{plain}
\bibliography{OBD.bib}

\end{document}